\documentclass[english]{article}
\usepackage{amsfonts}
\usepackage[latin9]{inputenc}
\usepackage{amsmath}
\usepackage{amssymb}
\usepackage{url}
\usepackage{graphicx}
\usepackage[pdfstartview=FitH]{hyperref}
\usepackage{amsthm}
\usepackage{authblk}
\makeatletter

\begin{document}
\newtheorem{theorem}{Theorem}[section]
\newtheorem{lemma}[theorem]{Lemma}
\newtheorem{definition}[theorem]{Definition}
\newtheorem{claim}[theorem]{Claim}
\newtheorem{example}[theorem]{Example}
\newtheorem{remark}[theorem]{Remark}
\newtheorem{proposition}[theorem]{Proposition}
\newtheorem{corollary}[theorem]{Corollary}

\title{Fixed point theorem for reflexive Banach spaces and uniformly convex non positively curved metric spaces}
\author{Izhar Oppenheim}
\affil{Department of Mathematics\\
 The Ohio State University  \\
 Columbus, OH 43210, USA \\
E-mail: izharo@gmail.com}

\maketitle
\textbf{Abstract}. This article generalizes the work of Ballmann  and \'{S}wiatkowski to the case of Reflexive Banach spaces and uniformly convex Busemann spaces, thus giving a new fixed point criterion for groups acting on simplicial complexes.   

\section{Introduction}
For a finite graph $(V,E)$ the Laplacian is a positive operator defined on functions $f:V \rightarrow \mathbb{R}$. One can generalize the definition of the Laplacian for a simplicial complex $X$ of any dimension and for such a complex the Laplacian is again a positive operator. \\ \\

Ballmann and \'{S}wiatkowski in \cite{BS} and independently \.{Z}uk's in \cite{Zuk} used the geometric information given by the Laplacian eigenvalues to give criteria
for the vanishing of cohomologies of a group $\Gamma$ acting on a simplicial complex. The most famous result of this type is the \.{Z}uk criterion which states that a group acting geometrically (i.e. cocompactly and proper discontinuously) on a $2$-dimensional simplicial complex has property (T) if the smallest positive Laplacian eigenvalues at the link of every vertex is large enough. It is well known that in the above setting, property (T) is equivalent to a fixed point property for action by isometries on a Real Hilbert space (see for instance \cite{BHV}). \\ \\

In this article, we generalize the \.{Z}uk criterion to reflexive Banach spaces and uniformly convex non positively curved Busemann spaces and get a fixed point criterion for those spaces relaying on the geometry of the links of vertices. The method that we use is basically taken from Gromov in \cite{Grom} (3.11), but we improve it so it doesn't require any scaling limit arguments (and generalize the form of the energy function). 

\textbf{Structure of the paper.} The first section gathers needed results about groups acting on simplicial complexes, Uniformly convex Busemann non positively curved spaces and Reflexive Banach spaces. The second section contains the main theorem and its proof.  

\section{Framework and Preliminaries}

\subsection{General Settings}

Throughout this paper $X$ is a simplicial complex of dimension $n \geq 2$ such that all the links of $X$ are connected and we assume that the links of all the vertices of $X$ are finite. Also $\Gamma$ is a locally compact, properly discontinuous, unimodular group of automorphisms of $X$ acting cocompactly on $X$.

Following \cite{BS} we introduce the following notations:
\begin{enumerate}
\item For $0\leq k\leq n$, denote by $\Sigma(k)$ the set of ordered $k$-simplices (i.e. $\sigma \in \Sigma(k)$ is and ordered $k+1$-tuple of vertices) and choose a set $\Sigma(k,\Gamma)\subseteq\Sigma(k)$ of representatives
of $\Gamma$-orbits.
\item For a simplex $\sigma\in\Sigma(k)$, denote by $\Gamma_{\sigma}$
the stabilizer of $\sigma$ and by $\vert \Gamma_{\sigma} \vert$ the measure of $\Gamma_{\sigma}$ with respect to the Haar measure.  

\end{enumerate}

The following proposition is taken from \cite{BS}, \cite{DJ1}:

\begin{proposition}
\cite[Lemma 1.3]{BS}, \cite[Lemma 3.3]{DJ1} For $0\leq l<k\leq n$, let $f=f(\tau,\sigma)$ be
a $\Gamma$-invariant function on the set of pairs $\left(\tau,\sigma\right)$,
where $\tau$ is an ordered $l$-simplex and $\sigma$ is an ordered
$k$-simplex with $\tau\subset\sigma$ Then \[
\sum_{\sigma\in\Sigma(k,\Gamma)}\sum_{\begin{array}{c}
{\scriptstyle \tau\in\Sigma(l)}\\
{\scriptstyle \tau\subset\sigma}\end{array}}\frac{f(\tau,\sigma)}{\left|\Gamma_{\sigma}\right|}=\sum_{\tau\in\Sigma(l,\Gamma)}\sum_{\begin{array}{c}
{\scriptstyle \sigma\in\Sigma(k)}\\
{\scriptstyle \tau\subset\sigma}\end{array}}\frac{f(\tau,\sigma)}{\left|\Gamma_{\tau}\right|}\]

\end{proposition}

The reader should note, that from now on we will use the above proposition to change the order of summation without mentioning it explicitly.

\begin{definition}
A weight on $X$ is an equivariant function $m : \bigcup_{2 \geq k \geq 1} \Sigma(k) \rightarrow \mathbb{R}^+_{>0}$ such that:
\begin{enumerate}
\item For every $\tau = (v_0,...,v_k)$ and for every permutation $\sigma \in S_k$ we have $m ((v_0,...,v_k)) = m ((v_{\sigma (0)},...,v_{\sigma (k)})$.
\item There is a $C (m)$ such that for every $\tau \in \Sigma (1)$ we have the following equality
$$ \sum_{\sigma \in \Sigma (2), \tau \subset \sigma} m( \sigma ) = 3! C (m) m (\tau )$$
Where $\tau \subset \sigma$ means that all the vertices of $\tau$ are contained in $\sigma$ (with no regard to the ordering). 
\end{enumerate} 
\end{definition}

\begin{example}
In \cite{BS} the function $m$ was defined as: for every $\tau \in \Sigma (k)$, $m (\tau)$ is the number of (unordered) simplices of dimension $n$ that contain $\tau$. In that case, $C (m) = n-1$.   
\end{example}

\begin{remark}
There is a lot of freedom in our definition of the weight function. Without loss of generality, one can always normalize the weight function such that $C (m) = 1$. It is obvious that in the normalized case the function $m$ is determined by its values on $\Sigma (2)$. We chose not to normalize the weight function in this paper as a matter of convenience and so that the reader could easily compare our results to those proven in \cite{BS}.  
\end{remark}

\begin{definition}
Let $u\in\Sigma(0)$, denote by $X_{u}$
the \emph{link} of $u$ in $X$, that is, the subcomplex of dimension
$n-1$ consisting on simplices $\sigma=\left(w_{0},...,w_{k}\right)$
such that $\left\{u \right\} ,\left\{ w_{0},...,w_{k}\right\}$ are disjoint as sets
and $\left(u,w_{0},...,w_{k}\right)=u\sigma\in\Sigma(k+1)$.
As stated above, $X$ is locally finite which means that $X_u$ is a finite simplicial complex.
\begin{enumerate}
\item For $0\leq k\leq n-1$, denote by $\Sigma_{u}(k)$ the set of
ordered $k$-simplices.
\item For a simplex $\sigma\in\Sigma_{u}(k)$ denote by $m_{u}(\sigma) = m (u \sigma )$.
\end{enumerate}
\end{definition}

\subsection{Uniformly convex Busemann non positively curved spaces}
In this subsection we will give definitions and some results about uniformly convex Busemann non positively curved spaces spaces. \\

Let $(Z,d)$ be a unique geodesic complete metric space, i.e. between any two points $x,y \in Z$ there is a unique geodesic connecting $x$ and $y$. For $x,y \in Z$ and $0 \leq t \leq 1$ denote by $tx + (1-t)y$ the point on the geodesic connecting $x$ and $y$ such that
$$ d(x,tx+(1-t)y)= t d(x,y), d(y,tx+(1-t)y)= (1-t) d(x,y)$$
This is of course only a notation because $Z$ need not be a vector space. 

\begin{definition}
A uniquely geodesic metric\footnote{Uniquely geodesic means that every two points have a unique geodesic connecting them. From now on we will assume that our spaces are always uniquely geodesic.}  space $(Z,d)$ will be called uniformly convex if there is a map $\theta : (0, \infty) \times (0,2] \rightarrow (0,1]$ such that for every $r>0$ and every $\varepsilon \in (0,2]$, for every three points $a,x,y \in Z$ the following implication holds:
$$ \begin{array}{c}
d(a,x) \leq r \\
d(a,y) \leq r \\
d(x,y) \geq \varepsilon r
\end{array}  \Rightarrow d (\frac{1}{2} x + \frac{1}{2} y ,a) \leq (1-\theta (r,\varepsilon ) )r$$
and
$$ \forall \varepsilon \in (0,2],\theta (\varepsilon ) = \inf \lbrace \theta (r,\varepsilon ) : r>0 \rbrace >0$$
The map $\theta$ is called the modulus of uniform convexity.
\end{definition}   

\begin{remark}
The above definition is closely related to \cite{GR}[page 107] (it appears also in \cite{RS} and \cite{GKM}). The reader should note there are other (non equivalent) definitions for uniformly convex metric space - see for instance \cite{KL}.
\end{remark}

Examples of uniformly convex metric spaces:
\begin{enumerate}
\item Hilbert spaces.
\item $L^p$ spaces for $1 < p < \infty$.
\item $CAT (0)$ spaces - the modulus of convexity might depend on $r$ but it is bounded by the modulus of convexity of a Hilbert space (for the same $\varepsilon$). 
\end{enumerate}

The following proposition is stated and proven in \cite{GR} for the Banach case (see \cite{GR}[Theorem 2.1], and it is also mentioned there that the proof is the same for the genera case of a uniformly convex metric space (see \cite{GR}[Theorem 18.1]). We shall repeat the proof here for completeness.

\begin{proposition}
Let $(Z,d)$ be a complete uniquely geodesic uniformly convex metric space, then for closed convex bounded non empty sets $C_n \subset Z$ such that $C_{n+1} \subset C_n$, we have $\bigcap_{n=1}^\infty C_n \neq \emptyset$.
\end{proposition}

\begin{proof}
Take arbitrary $x \in Z$. If $x \in \bigcap_{n=1}^\infty C_n$ we are done. If not there is some $N$, such that $d(x,C_N) >0$ (recall that $C_N$ is closed). Denote $r_n = d(x,C_n)$, then $\lbrace r_n \rbrace$ is an increasing non negative sequence which is bounded from above because $C_1$ is bounded. Denote $r = \lim r_n \geq r_N >0$. Define $D_n = C_n \cap \overline{B}(x, r + \frac{1}{n})$, by completeness it is enough to show that $diam (D_n) \rightarrow 0$, because then $\bigcap_{n=1}^\infty C_n \supseteq \bigcap_{n=1}^\infty D_n \neq \emptyset$. Assume toward contradiction that $diam (D_n) \rightarrow d >0$, then there is some $n_0$ such that for every $n > n_0$ we have $\frac{1}{n} < \frac{d}{2}$. For every $n > n_0$ we have points $x_n,y_n \in D_n$ such that 
$$d(x_n,y_n) > d - \dfrac{1}{n} > \dfrac{d}{2} = \dfrac{d}{2 (r+\dfrac{1}{n})}(r+\dfrac{1}{n}) \geq \dfrac{d}{2 (r+1)}(r+\dfrac{1}{n}) $$
and since $x_n,y_n \in D_n$ we have $d(x,x_n), d(x,y_n) \leq r + \frac{1}{n}$. 
Note that $\frac{d}{2} \leq r+1$ because $D_n \subseteq \overline{B}(x,r+1)$ and therefore $\frac{d}{2(r+1)} \leq 1$. By uniform convexity, for every $n > n_0$ we have
$$r_n \leq d(x, \frac{1}{2} x_n + \frac{1}{2} y_n ) \leq \left( 1- \theta (r+\frac{1}{n}, \dfrac{d}{2 (r+1)}) \right) (r+ \frac{1}{n}) \leq$$
$$ \leq \left( 1- \theta ( \dfrac{d}{2 (r+1)}) \right) (r+ \frac{1}{n})$$
we can take $n \rightarrow \infty$ and get that
$$ r \leq \left( 1- \theta (\dfrac{d}{2 (r+1)}) \right) r < r$$
which is a contradiction.
\end{proof}

The above proposition has a useful corollary (also taken from \cite{GR}[Proposition 18.2] with minor adaptations) :

\begin{corollary}
Let $(Z,d)$ be a complete uniquely geodesic uniformly convex metric space and let $f : Z \rightarrow \mathbb{R}^+$ be a quasi-convex function, i.e. 
$$ \forall 0 \leq t \leq 1, \forall x,y \in Z, f(tx+ (1-t)y) \leq \max \lbrace f(x), f(y) \rbrace$$
If there is a $c \in \mathbb{R}^+$ such that the set $\lbrace x : f(x) \leq c \rbrace$ is non empty and bounded, then $f$ has a minimum. Moreover, if $f$ is strictly quasi-convex, i.e.  
$$ \forall x,y \in Z,x \neq y, f(\dfrac{1}{2}x+ \dfrac{1}{2}y) < \max \lbrace f(x),f(y) \rbrace$$
then this minimum is unique.
\end{corollary} 

\begin{proof}
Denote $c' = \inf \lbrace f(x) : x \in Z \rbrace \geq 0$, we shall show that $c'$ is the minimum of $f$. If $c' = c$ we are done because we know that $\lbrace x : f(x) \leq c \rbrace$ is non empty. Otherwise, there is an integer $n_0$ such that $\frac{1}{n_0} < c-c'$. Then for every $n \geq n_0$, define 
$$C_n = \lbrace x : f(x) \leq c'+\frac{1}{n} \rbrace$$
Those are bounded non empty convex sets (because $f$ is quasi-convex) such that $C_{n+1} \subset C_n$ and by the above proposition we get that $\bigcap_{n=n_0}^\infty C_n \neq \emptyset$ and for $x \in \bigcap_{n=n_0}^\infty C_n$ we get that $f(x) = c'$. If $f$ is strictly quasi-convex then for every two $x,y \in Z$ such that $f(x)=f(y)=c'$ we get that if $x \neq y$ then  
 $$f(\dfrac{1}{2}x+ \dfrac{1}{2}y) < \max \lbrace f(x),f(y) \rbrace = c'$$
which is a contradiction of the definition of $c'$ as the infimum. 
\end{proof}
 
\begin{definition}
A uniquely geodesic metric space $(Z,d)$ is said to be non positively curved in the sense of Busemann, if for every three points: $x,y,z \in Z$ one has
$$ d(\dfrac{1}{2}x+ \dfrac{1}{2}z,\dfrac{1}{2}y+ \dfrac{1}{2}z ) \leq \dfrac{1}{2} d(x,y)$$
\end{definition}

\begin{remark}
Some papers use the term hyperbolic for metric spaces that are non positively curved in the sense of Busemann (see \cite{RS}[Definition 2.1]). We use the term "non positively curved in the sense of Busemann" to avoid confusion with other non equivalent terms as Gromov hyerbolic metric spaces.
\end{remark}

Observe that the condition stated it the definition above is equivalent to the condition:
$$ \forall x, y, x',y' \in Z, d(\dfrac{1}{2}x+ \dfrac{1}{2}y,\dfrac{1}{2}x' + \dfrac{1}{2}y') \leq \dfrac{1}{2} d(x,y')  + \dfrac{1}{2} d(x',y)$$  

\begin{proposition}
If $(Z,d)$ uniquely geodesic metric space which is non positively curved in the sense of Busemann and uniformly convex, then for every $y \in Z$, the function $d(.,y) : Z \rightarrow \mathbb{R}$ is a convex function.  
\end{proposition}

\begin{proof}
Let $y,x_1,x_2 \in Z$, we need to show that 
$$  d(\dfrac{x_1+x_2}{2},y) \leq \frac{1}{2} d(x_1,y)+ \frac{1}{2} d(x_2,y)$$
Assume WLOG that $d(x_1,y) \geq d(x_2,y)$, then on the geodesic connecting $y$ and $x_1$ there is a point $x_1'$ s.t. $d(x_1',y)=d(x_2,y)$. From uniform convexity we get that
$$ d(\dfrac{x_1'+x_2}{2},y) \leq d(x_2,y)$$
(Note this need not be a strict inequality because we might have $x_1'=x_2$). \\
From the non positive curvature we get that
$$ d(\dfrac{x_1'+x_2}{2},\dfrac{x_1+x_2}{2}) \leq \frac{1}{2} d(x_1,x_1') = \frac{1}{2} (d(x_1,y)-d(x_2,y))$$
Therefore
$$d(\dfrac{x_1+x_2}{2},y) \leq d(\dfrac{x_1'+x_2}{2},y) + d(\dfrac{x_1'+x_2}{2},\dfrac{x_1+x_2}{2}) \leq$$
$$\leq d(x_2,y) + \frac{1}{2} (d(x_1,y)-d(x_2,y)) = \frac{1}{2} d(x_1,y)+ \frac{1}{2} d(x_2,y)$$    
\end{proof}

Last, observe that if $(Z,d)$ is uniquely geodesic, then for every isometry $T$ of $Z$ and for every two points $x,y \in Z$, one has   
$$T(\dfrac{1}{2} x + \dfrac{1}{2} y) = \dfrac{1}{2} T(x) + \dfrac{1}{2} T(y)$$
 
\subsection{Reflexive Banach spaces}
In this subsection we will recall some facts about reflexive Banach spaces which will be very similar to the facts we recalled in the previous section. 

\begin{proposition}
Let $(Z,\vert . \vert)$ be a reflexive Banach space, then for closed bounded non empty sets $C_n \subset Z$ such that $C_{n+1} \subset C_n$, we have $\bigcap_{n=1}^\infty C_n \neq \emptyset$.
\end{proposition}

\begin{proof}
Every set $C_n$ is closed and bounded and therefore is compact in the weak topology (because $Z$ is reflexive) and therefore $\bigcap_{n=1}^\infty C_n \neq \emptyset$.
\end{proof}

As in the previous subsection we get the following corollary (the proof is exactly the same): 
\begin{corollary}
Let $(Z,\vert . \vert)$ be a reflexive Banach space and let $f : Z \rightarrow \mathbb{R}^+$ be a quasi-convex function, i.e. 
$$ \forall 0 \leq t \leq 1, \forall x,y \in Z, f(tx+ (1-t)y) \leq \max \lbrace f(x), f(y) \rbrace$$
If there is a $c \in \mathbb{R}^+$ such that the set $\lbrace x : f(x) \leq c \rbrace$ is non empty and bounded, then $f$ has a minimum. Moreover, if $f$ is strictly quasi-convex, i.e.  
$$ \forall x,y \in Z,x \neq y, f(\dfrac{1}{2}x+ \dfrac{1}{2}y) < \max \lbrace f(x),f(y) \rbrace$$
then this minimum is unique.
\end{corollary} 

Also observe that for every Banach space $(Z,\vert . \vert)$ we have 
$$ \forall x, y, x',y' \in Z, d(\dfrac{1}{2}x+ \dfrac{1}{2}y,\dfrac{1}{2}x' + \dfrac{1}{2}y') \leq \dfrac{1}{2} d(x,y')  + \dfrac{1}{2} d(x',y)$$
(Where $d$ is the usual metric induced by the norm) and that for every $y \in Z$, $d(.,y): Z \rightarrow \mathbb{R}$ is a convex function.

Finally, recall the Mazur-Ulam theorem (see for instance \cite{FJ}):
\begin{theorem}
Every surjective isometery between normed spaces is affine. 
\end{theorem}
Which yield that for every isometry $T$ of $(Z, \vert . \vert)$ and for any two points $x,y \in Z$ we have
$$T(\dfrac{1}{2} x + \dfrac{1}{2} y) = \dfrac{1}{2} T(x) + \dfrac{1}{2} T(y)$$

\subsection{Uniformly convex Busemann non positively curved spaces and Reflexive Banach spaces concluded}
In this subsection we conclude the mutual facts gathered in the last two subsections: let $(Z,d)$ be a uniformly convex, uniquely geodesic, Busemann non positively curved, complete metric space or a reflexive Banach space (where $d$ is the metric induced by the norm), then the following holds:
\begin{itemize}
\item For every strictly convex function $f : Z \rightarrow \mathbb{R}^+$ 
if there is a $c \in \mathbb{R}^+$ such that the set $\lbrace x : f(x) \leq c \rbrace$ is non empty and bounded, then $f$ has a unique minimum.
\item  $$ \forall x, y, x',y' \in Z, d(\dfrac{1}{2}x+ \dfrac{1}{2}y,\dfrac{1}{2}x' + \dfrac{1}{2}y') \leq \dfrac{1}{2} d(x,y')  + \dfrac{1}{2} d(x',y)$$
\item For every $y \in Z$, $d(.,y): Z \rightarrow \mathbb{R}$ is a convex function.
\item For every isometry $T$ on $Z$ and for every two points $x,y \in Z$, we have 
$$T(\dfrac{1}{2} x + \dfrac{1}{2} y) = \dfrac{1}{2} T(x) + \dfrac{1}{2} T(y)$$
\end{itemize}  

Note that the summation symbols means two different things: in the uniquely geodesic metric space, $\dfrac{1}{2} x + \dfrac{1}{2} y$ means the unique midpoint between $x$ and $y$ and in the Banach case $\dfrac{1}{2} x + \dfrac{1}{2} y$ means the average of the two vectors (which is a midpoint, but it need not be unique).
From now on we will consider $(Z,d)$ to be either uniformly convex, uniquely geodesic, Busemann non positively curved, complete metric space or a reflexive Banach space and we will use only the mutual facts stated above (and the completeness).

\section{Fixed point criteria via links}
Let $\Gamma$ be a group acting on an $n$-dimentional simplicial complex as above, let $(Z,d)$ be either uniformly convex, uniquely geodesic, Busemann non positively curved, complete metric space or a reflexive Banach space and let $\rho : \Gamma \rightarrow Isom(Z)$. Fix a function $f: \mathbb{R}^+ \rightarrow \mathbb{R}^+$ with the following properties:
\begin{itemize}
\item $f(0)=0$ and $f$ is strictly monotone increasing.
\item $f$ is strictly convex (and therefore $\lim_{x \rightarrow \infty} f(x) = \infty$).
\item For every constants $0 \leq \kappa < 1, C \geq 0$ we have $\sum_{k=1}^\infty f^{-1} (C \kappa^k) < \infty$
\end{itemize} 
Examples for such functions are $f(x)=x^p$ with $p>1$ and $f(x)=a_2 x^2+...+a_k x^k$ with $a_2,...,a_k$ positive. \\

For every vertex $u \in X$ denote 
$$C^0 (X_u)= \lbrace \phi : \Sigma_u (0) \rightarrow Z \rbrace$$
and define $E_{u,\phi} : Z \rightarrow \mathbb{R}^+$ as
$$E_{u,\phi} (\xi) = \sum_{v \in \Sigma_u (0)} m((u,v)) f(d(\xi,\phi (v)))$$
Note that $E_{u,\phi}$ is strictly convex, because $f$ is strictly convex and $d(.,\phi (v))$ is convex for every $v$. Since $E_{u,\phi} (\xi) \rightarrow \infty$ as $d(Im (\phi),\xi) \rightarrow \infty$ we get that $E_{u,\phi}$ has a unique minimum. Therefore there is a map $M_u : C^0 (X_u (0)) \rightarrow Z$ which send each $\phi$ to the minimum of $E_{u,\phi}$.  \\
For every $\phi \in C^0 (X_u)$ define $d_u \phi : \Sigma_u (1) \rightarrow \mathbb{R}$ to be $d_u \phi (v,w) = d (\phi (v),\phi (w))$. \\

Define for every vertex $u \in X$ a constant $\lambda_u$ as following:
$$\lambda_u = \sup \lbrace \lambda : \lambda E_{u,\phi} (M_u \phi) \leq \sum_{ \eta \in \Sigma_u (1)} \frac{m_u (\eta)}{2}  f(d_u \phi (\eta)), \forall \phi \in C^0 (X_u) \rbrace$$ 
 
Denote 
$$ C^0 (X,\rho) = \lbrace \phi : \Sigma (0) \rightarrow X: \phi \text{ is equivariant w.r.t } \rho \rbrace$$
Define an operator $M: C^0 (X,\rho) \rightarrow C^0 (X,\rho)$ as
$$ \forall u \in \Sigma (0), M \phi (u) = M_u \phi \vert_{X_u}$$
Where $\phi \vert_{X_u}$ is the restriction of $\phi$ to the link of $u$. \\
For $\phi, \phi ' \in  C^0 (X,\rho)$ define $\frac{1}{2} \phi + \frac{1}{2} \phi ' \in C^0 (X,\rho)$ as 
$$ \forall u \in \Sigma (0), (\frac{1}{2} \phi + \frac{1}{2} \phi ' ) (u) = \frac{1}{2} \phi (u) + \frac{1}{2} \phi ' (u)$$
Define the operator $M' : C^0 (X,\rho) \rightarrow C^0 (X,\rho)$ as
$$M' \phi = \frac{1}{2} \phi + \frac{1}{2} M \phi $$  

\begin{proposition}
The images of the operators $M$ and $M'$ are indeed contained in $C^0 (X,\rho)$.     
\end{proposition}

\begin{proof}
To show that the image of $M$ is contained in $C^0 (X,\rho)$, we need to show that for every $\phi \in C^0 (X,\rho)$ we have that $M \phi$ is an equivariant map w.r.t $\rho$, i.e. for every $u \in \Sigma (0)$ and every $\gamma \in \Gamma$ we have that 
$$ \rho (\gamma) . M \phi (u) = M \phi (\gamma . u)$$
Fix some $u \in \Sigma (0)$ and $\gamma \in \Gamma$, then $\gamma$ take the link of $u$ to the link of $\gamma . u$ and since $\phi$ is equivariant and the weight $m$ is invariant we get that
$$E_{\gamma . u,\phi \vert_{X_{\gamma . u}}} (\rho (\gamma) . \xi) = \sum_{\gamma . v \in \Sigma_{\gamma . u} (0)} m((\gamma . u,\gamma . v)) f(d(\rho (\gamma) . \xi,\phi \vert_{X_{\gamma . u}}  (\gamma . v)))=$$
$$= \sum_{\gamma . v \in \Sigma_{\gamma . u} (0)} m((u,v)) f(d(\rho (\gamma) . \xi,\rho (\gamma) . \phi \vert_{X_u} (v)))=$$
$$= \sum_{\gamma . v \in \Sigma_{\gamma . u} (0)} m((u,v)) f(d(\xi,\phi (v)))= E_{u,\phi} (\xi)$$
and therefore if $M_u \phi \vert_{X_u}$ is the unique minimum of $E_{u,\phi \vert_{X_u}} (\xi)$ then $\rho (\gamma) . M_u \phi \vert_{X_u}$ is the unique minimum of $E_{\gamma . u, \phi \vert_{X_{\gamma . u}}} ( \xi)$ and the map $M$ is equivariant. \\
To show $M'$ is equivariant we simply recall that for every isometry $T$ of $Z$ and any $x,y \in Z$, we have 
$$T (\frac{1}{2} x + \frac{1}{2} y) = \frac{1}{2} T(x) + \frac{1}{2} T(y)$$
\end{proof}

Define $E(.,.) : C^0 (X,\rho) \times C^0 (X,\rho) \rightarrow \mathbb{R}$ as 
$$ E(\phi,\psi) = \sum_{(u,v) \in \Sigma (1,\Gamma)} \dfrac{m((u,v))}{ \vert \Gamma_{(u,v)} \vert} f (d(\phi (u), \psi (v)))$$

\begin{proposition}
\begin{enumerate}
\item For every $\phi, \psi \in C^0 (X,\rho)$ we have
$$E(\phi,\psi) = \sum_{u \in \Sigma (0,\Gamma)} \dfrac{1}{\vert \Gamma_{u} \vert} E_{u,\phi} (\psi (u))$$
\item For every $\phi \in C^0 (X,\rho)$ we have
$$\dfrac{C (m)}{2} E (\phi, \phi ) = \sum_{u \in \Sigma (0,\Gamma)} \dfrac{1}{\vert \Gamma_{u} \vert} \sum_{ \eta \in \Sigma_u (1)} \frac{m_u (\eta)}{2}  f(d_u \phi (\eta))$$
Where $C (m)$ is the constant such that for every $\eta \in \Sigma (1)$ we have
$$\sum_{ \sigma \in \Sigma (2), \eta \subset \sigma} m (\sigma) = 3! C (m) m( \eta)$$
\item For every $\phi, \phi ' , \psi , \psi ' \in C^0 (X,\rho)$ we have
$$ E (\frac{1}{2} \phi + \frac{1}{2} \phi ', \frac{1}{2} \psi + \frac{1}{2} \psi ') \leq \frac{1}{2} E (\phi, \psi ') + \frac{1}{2} E (\phi ', \psi)$$
\end{enumerate}
\end{proposition}

\begin{proof}
\begin{enumerate}
\item For every $\phi, \psi \in C^0 (X,\rho)$ we have
$$\sum_{u \in \Sigma (0,\Gamma)} \dfrac{1}{\vert \Gamma_{u} \vert} E_{u,\phi} (\psi (u)) = $$
$$=\sum_{u \in \Sigma (0,\Gamma)} \dfrac{1}{\vert \Gamma_{u} \vert} \sum_{v \in \Sigma_u (0)} m((u,v)) f(d(\psi (u),\phi (v)))=$$
$$=  \sum_{u \in \Sigma (0,\Gamma)} \dfrac{1}{\vert \Gamma_{u} \vert} \sum_{\eta \in \Sigma (1), u \subset \eta}\frac{m(\eta)}{2}  f(d(\psi (u),\phi (\eta - u)))  $$
Where $\eta - u$ is $v$ for $\eta = (u,v)$ or for $\eta = (v,u)$ (hence the division by $2$). Changing the order of summation gives: 
$$ \sum_{\eta \in \Sigma (1,\Gamma)} \dfrac{m(\eta)}{2 \vert \Gamma_{\eta} \vert} \sum_{u \in \Sigma (0), u \subset \eta}  f(d (\psi (u),\phi (\eta - u))) = $$
$$ =\sum_{(u,v) \in \Sigma (1,\Gamma)} \dfrac{m((u,v))}{2 \vert \Gamma_{(u,v)} \vert} ( f(d (\psi (u),\phi (v))) + f(d (\psi (v),\phi (u)))=$$
$$= \sum_{(u,v) \in \Sigma (1,\Gamma)} \dfrac{m((u,v))}{ \vert \Gamma_{(u,v)} \vert}  f(d (\psi (u),\phi (v)))= E(\phi,\psi)$$

\item For every $\phi \in C^0 (X,\rho)$ we have
$$\sum_{u \in \Sigma (0,\Gamma)} \dfrac{1}{\vert \Gamma_{u} \vert} \sum_{ \eta \in \Sigma_u (1)} \frac{m_u (\eta)}{2}  f(d_u \phi (\eta)) = $$
$$ \sum_{u \in \Sigma (0,\Gamma)} \dfrac{1}{\vert \Gamma_{u} \vert} \sum_{ \sigma \in \Sigma (2), u \subset \sigma} \frac{m (\sigma)}{6}  f(d_u \phi (\sigma - u))$$
Where again $\sigma - u = (v,w)$ for $\sigma = (u,v,w),(v,u,w), (v,w,u)$. Changing the order of summation gives
$$ \sum_{\sigma \in \Sigma (2,\Gamma)} \dfrac{m (\sigma)}{6 \vert \Gamma_{\sigma} \vert} \sum_{ u \in \Sigma (0), u \subset \sigma} f(d_u \phi (\sigma - u))=$$
$$ \sum_{\sigma \in \Sigma (2,\Gamma)} \dfrac{m (\sigma)}{6 \vert \Gamma_{\sigma} \vert} \sum_{ \eta \in \Sigma (1), \eta \subset \sigma} \dfrac{1}{2}f(d \phi (\eta))$$
Where $d \phi ((v,w)) = d(\phi (v),\phi (w)$ and the factor $\frac{1}{2}$ is because $(v,w),(w,v) \subset (u,v,w)$. Again we can change the order of summation and get
$$ \sum_{\eta \in \Sigma (1,\Gamma)} \dfrac{f(d \phi (\eta))}{12 \vert \Gamma_{\eta} \vert} \sum_{ \sigma \in \Sigma (2), \eta \subset \sigma} m (\sigma) =$$
$$ \sum_{\eta \in \Sigma (1,\Gamma)} \dfrac{C (m) m( \eta ) f(d \phi (\eta))}{2 \vert \Gamma_{\eta} \vert}= \dfrac{C (m)}{2} E (\phi , \phi)$$
\item For every  $\phi, \phi ' , \psi , \psi ' \in C^0 (X,\rho)$ we have
$$ E (\frac{1}{2} \phi + \frac{1}{2} \phi ', \frac{1}{2} \psi + \frac{1}{2} \psi ') =  \sum_{(u,v) \in \Sigma (1,\Gamma)} \dfrac{m((u,v))}{ \vert \Gamma_{(u,v)} \vert} f (d(\frac{1}{2} \phi (u) + \frac{1}{2} \phi ' (u), \frac{1}{2} \psi (v) + \frac{1}{2} \psi ' (v))) \leq $$
$$  \leq \sum_{(u,v) \in \Sigma (1,\Gamma)} \dfrac{m((u,v))}{ \vert \Gamma_{(u,v)} \vert} f (\frac{1}{2} d( \phi (u),\psi ' (v)) + \frac{1}{2} d(\phi ' (u), \psi (v)))$$
Where the inequality follows for properties of $Z$ and for the fact that $f$ is monotone increasing. From convexity of $f$ we get
$$ \leq  \sum_{(u,v) \in \Sigma (1,\Gamma)} \dfrac{m((u,v))}{ \vert \Gamma_{(u,v)} \vert}( \frac{1}{2} f ( d( \phi (u),\psi ' (v)) + \frac{1}{2}  f(d(\phi ' (u), \psi (v)))) = $$
$$ \frac{1}{2} E (\phi, \psi ') + \frac{1}{2} E (\phi ', \psi)$$
\end{enumerate}
\end{proof}

Denote $\lambda = min \lbrace \lambda_u : u \in \Sigma (0) \rbrace$, then from the above proposition we get:

\begin{corollary}
\begin{enumerate}
\item For every $\phi \in C^0 (X,\rho)$ we have that
$$ \dfrac{C (m)}{2} E (\phi , \phi ) \geq \lambda E (M \phi , \phi )$$
\item For every $\phi \in C^0 (X,\rho)$ we have that
$$ \dfrac{C (m)}{2} E (\phi , \phi ) \geq \lambda E (M' \phi ,M' \phi )$$
\end{enumerate}
\end{corollary} 

\begin{proof}
\begin{enumerate}
\item For every $\phi \in C^0 (X,\rho)$, by 2. in the above proposition we get that
$$\dfrac{C (m)}{2} E (\phi, \phi ) = \sum_{u \in \Sigma (0,\Gamma)} \dfrac{1}{\vert \Gamma_{u} \vert} \sum_{ \eta \in \Sigma_u (1)} \frac{m_u (\eta)}{2}  f(d_u \phi (\eta)) \geq $$
$$ \geq \lambda \sum_{u \in \Sigma (0,\Gamma)} \dfrac{1}{\vert \Gamma_{u} \vert}  E_{u,\phi} (M_u \phi)\geq \lambda E (\phi , M \phi)$$ 
Where the first inequality is due to the definition of $\lambda$ and the second inequality is due to 1. in the above proposition. 
\item Due to 3. in the above proposition we get that
$$ \lambda E (M' \phi ,M' \phi ) =\lambda E (\frac{1}{2} \phi + \frac{1}{2} M \phi  ,\frac{1}{2} \phi + \frac{1}{2} M \phi ) \leq $$
$$ \leq \frac{1}{2} \lambda E (\phi, M \phi) + \frac{1}{2} \lambda E (M \phi, \phi) \leq \dfrac{C (m)}{2} E (\phi, \phi )$$ 

\end{enumerate}
\end{proof}

\begin{theorem}
If $\lambda > \dfrac{C (m)}{2}$ then $\Gamma$ has a fixed point for every $\rho$.
\end{theorem}

\begin{proof}
Denote by $\kappa = \dfrac{C (m)}{2 \lambda}$, then $0 \leq \kappa <1$ and for every $\phi \in C^0 (X,\rho)$ 
$$ \kappa E (\phi, \phi) \geq E (M' \phi,M' \phi)$$
Therefore, for every $k \in \mathbb{N}$ we have
$$ \kappa^k E (\phi, \phi) \geq E ((M')^k \phi, (M')^k \phi)$$
Denote by $\delta = min \lbrace m (\eta ) : \eta \in \Sigma (1) \rbrace > 0$, then for every $(u,v) \in \Sigma (1,\Gamma)$ we have
$$ \kappa^k E (\phi, \phi) \geq E ((M')^k \phi, (M')^k \phi) \geq \delta f (d((M')^k \phi (u), d((M')^k \phi (v)))$$ 
and therefore, for every $(u,v) \in \Sigma (1,\Gamma)$ we have 
$$ f^{-1} (\dfrac{\kappa^k E (\phi, \phi)}{\delta} ) \geq d((M')^k \phi (u), d((M')^k \phi (v))$$
We also have for every $\phi \in C^0 (X,\rho)$ that
$$ \kappa E (\phi, \phi) \geq E (M \phi, \phi)$$
and therefore 
$$ E ((M')^k \phi, (M')^k \phi) > \kappa E ((M')^k \phi, (M')^k \phi) \geq E ((M')^k \phi, M (M')^k \phi)$$
So we have
$$ E ((M')^k \phi, (M')^{k+1} \phi) = E ((M')^k \phi, \frac{1}{2}(M')^{k} \phi + \frac{1}{2} M(M')^{k} \phi) \leq $$
$$ \leq  \frac{1}{2} E ((M')^k \phi, (M')^{k} \phi )+ \frac{1}{2} E((M')^k \phi, M(M')^{k} \phi) < E ((M')^k \phi, (M')^k \phi)$$
Therefore, for every $(u,v) \in \Sigma (1,\Gamma)$ we have (as before)
$$ f^{-1} (\dfrac{\kappa^k E (\phi, \phi)}{\delta} ) \geq d((M')^k \phi (u), (M')^{k+1} \phi (v))$$
By triangle inequality we have for every $u \in \Sigma (0)$
$$ 2 f^{-1} (\dfrac{\kappa^k E (\phi, \phi)}{\delta} ) \geq d((M')^k \phi (u), (M')^{k+1} \phi (u))$$
Now since 
$$\sum_{k=1}^\infty 2 f^{-1} (\dfrac{\kappa^k E (\phi, \phi)}{\delta} ) < \infty$$
Then for every $u \in \Sigma (0)$, $(M')^k \phi (u)$ is Cauchy sequence and therefore we can define $\phi_0 \in C^0 (X,\rho)$ as 
$$\forall u \in \Sigma (0), \phi_0 (u) = \lim (M')^k \phi (u)$$
We get that $E (\phi_0, \phi_0) = 0$ and therefore $\phi_0$ must be a constant, equivariant map (so there is a fixed point).
\end{proof}

\begin{corollary}
The above theorem generalizes several previous theorems:
\begin{enumerate}
\item For the case $f(x)=x^2$ and $Z$ is a Hilbert space, we get the criterion stated in \cite{BS} and \cite{Zuk} (this is the famous \.{Z}uk criterion).
\item For the case $f(x)=x^2$ and $Z$ is a Hadamard space (i.e. CAT(0) and complete), we get the criterion stated in \cite{IN} (for the $2$ dimensional case).
\item For the case $f(x)=x^p$ and $Z$ is $L^p$ for $1<p<\infty$, we get the criterion in \cite{Bou} (for the $2$ dimensional case).
\item For the case $f(x)=x^p$ and $Z$ is a reflexive Banach space, we improve the criterion in \cite{Now}.

\end{enumerate}
\end{corollary}

\begin{remark}
Note that in the case that $Z$ is a Banach space, the reflexivity of $Z$ and the strict convexity of $f$ were only required to define $M$ uniquely and to insure it is an equivariant operator. We can avoid those restrictions and work in a general Banach space, if we define a different equivariant operator $M$. For instance, if we define 
$$M \phi_u = \sum_{v \in \Sigma_u (0)} \dfrac{m_u (v) }{\sum_{v \in \Sigma_u (0)} m_u (v)} \phi (v)$$
Then be Mazur-Ulam theorem, $M$ will be an equivariant operator and we have the same criterion with the appropriate $\lambda$.     
\end{remark}

\bibliographystyle{alpha}
\bibliography{bibl}
 
\end{document}